\documentclass[11pt]{amsart}

\usepackage[english]{babel}
\usepackage[utf8]{inputenc}
\usepackage{amsmath, amssymb, amsthm}
\usepackage{graphicx}
\usepackage{hyperref}
\usepackage{caption}
\usepackage{array}
\usepackage{stmaryrd}
\usepackage{longtable}
\usepackage{booktabs}
\usepackage{makecell}

\hypersetup{
	colorlinks=true,
	urlcolor=blue,      
	citecolor=red,      
	linkcolor=red       
}

\theoremstyle{definition}

\newtheorem{teo}{Theorem}

\newtheorem{ex}{Example}
\newtheorem{defi}{Definition}

\newtheorem{prop}{Proposition}

\newtheorem{lema}{Lemma}

\usepackage{tikz}
\usepackage{xcolor}

\definecolor{bloco1}{RGB}{230,240,255}
\definecolor{bloco2}{RGB}{255,240,230}
\definecolor{pista}{RGB}{10,10,120}
\definecolor{simetria}{RGB}{200,0,0}

\newcommand{\sudoku}[1]{
	\begin{tikzpicture}[scale=0.5]
		\draw[step=1,thin] (0,0) grid (9,9);
		\draw[step=3,very thick] (0,0) grid (9,9);
		
		\foreach \x [count=\i] in {#1} {
			\pgfmathtruncatemacro{\col}{mod(\i-1,9)}
			\pgfmathtruncatemacro{\row}{8-floor((\i-1)/9)}
			
			\def\zero{0}
			\ifx\x\zero
			\else
			\node at (\col+0.5,\row+0.5) {\x};
			\fi
		}
	\end{tikzpicture}
}

\newcommand{\blocossudoku}{
	\begin{tikzpicture}[scale=0.5]
		
		\draw[step=1,thin,gray!30] (0,0) grid (9,9);
		
		\draw[step=3,very thick] (0,0) grid (9,9);
		
		\node[scale=2] at (1.5,7.5) {$B_1$};
		\node[scale=2] at (4.5,7.5) {$B_2$};
		\node[scale=2] at (7.5,7.5) {$B_3$};
		
		\node[scale=2] at (1.5,4.5) {$B_4$};
		\node[scale=2] at (4.5,4.5) {$B_5$};
		\node[scale=2] at (7.5,4.5) {$B_6$};
		
		\node[scale=2] at (1.5,1.5) {$B_7$};
		\node[scale=2] at (4.5,1.5) {$B_8$};
		\node[scale=2] at (7.5,1.5) {$B_9$};
		
	\end{tikzpicture}
}

\newcommand{\sudokubanda}[1]{
	\begin{tikzpicture}[scale=0.5]
		\draw[step=1,thin] (0,0) grid (9,3);
		\draw[very thick] (0,0) rectangle (9,3);
		\foreach \x in {3,6} {
			\draw[very thick] (\x,0) -- (\x,3);
		}
		
		\foreach \x [count=\i] in {#1} {
			\pgfmathtruncatemacro{\col}{mod(\i-1,9)}
			\pgfmathtruncatemacro{\row}{2-floor((\i-1)/9)}
			
			\def\zero{0}
			\ifx\x\zero
			\else
			\node at (\col+0.5,\row+0.5) {\x};
			\fi
		}
	\end{tikzpicture}
}

\newcommand{\sudokubloco}{
\begin{tikzpicture}[scale=0.5]
	
	\draw[step=1,thin] (0,0) grid (3,3);
	\draw[very thick] (0,0) rectangle (3,3);
	
	\foreach \x [count=\i] in {1,2,3,4,5,6,7,8,9}{
		\pgfmathtruncatemacro{\col}{mod(\i-1,3)}
		\pgfmathtruncatemacro{\row}{2-floor((\i-1)/3)}
		\node at (\col+0.5,\row+0.5) {\x};
	}
	
\end{tikzpicture}
}

\title{Counting, Symmetries and Equivalence Classes of Sudoku Grids}
\author{Fernanda Pereira}
\thanks{Instituto Tecnológico de Aeronáutica (ITA), São José dos Campos-SP, Brazil. E-mail: \texttt{fpereira@ita.br}}
\date{}

\begin{document}

\maketitle

\begin{abstract}
	Sudoku is a widely popular puzzle whose complete grids have been enumerated computationally: there are approximately $6.67 \times 10^{21}$ of them and, up to symmetry and relabeling of digits, $5,472,730,538$ essentially different ones. The classical enumeration reduces the count to $44$ equivalence classes of the first band through a chain of ad hoc reductions, leaving the number $44$ without any apparent structural explanation. We present an alternative derivation of these $44$ classes, in which they arise as isomorphism classes of unordered triples (multisets) of column partitions under relabeling, a single invariant that replaces the original chain of reductions. This invariant makes it possible to apply Burnside's Lemma by hand: we recover $44$ through a closed derivation requiring no computational enumeration.
\end{abstract}


\section{Introduction}

Classic Sudoku is a puzzle defined on a square grid made up of 9 rows and 9 columns, totaling 81 cells, which are also organized into 9 blocks of dimension $3 \times 3$. The goal is to fill all the cells with digits from 1 to 9, so that each number occurs exactly once in each row, in each column, and in each $3 \times 3$ block. Each puzzle presents some cells that are already filled in, called clues, from which the rest of the grid must be completed while respecting the imposed constraints. When the initial configuration admits a unique solution, the Sudoku is said to be well-posed or valid.

Sudoku descends from the Latin squares studied by Euler, and reached its current form only in 1979, through Howard Garns's Number Place, later popularized in Japan by the publisher Nikoli under the name Sudoku and adopted worldwide by newspapers and puzzle books \cite{Delahaye2006, Hayes2006}.

Beyond its popularity as a puzzle, Sudoku is a rich combinatorial object and the subject of a substantial body of mathematical work, touching combinatorics, group theory, graph theory, and computational complexity: its underlying logic has been studied in depth \cite{Berthier2007, Rosenhouse2011}, it has been modeled as a graph-coloring problem, and the minimum number of clues admitting a unique solution has been settled computationally.

The enumeration of complete Sudoku grids has historically proceeded through an ad hoc chain of reductions that arrives at the number $44$ (the count of essentially different first bands) as an accumulation of successive tricks, leaving the impression that $44$ has no structural meaning of its own. The count of essentially different grids as a whole, in turn, has so far relied entirely on a computational application of Burnside's Lemma.

This paper addresses both counting questions. Section~\ref{sec:contagem} presents the classical enumeration of complete grids by Felgenhauer and Jarvis \cite{FelgenhauerJarvis2006}. Section~\ref{sec:derivacao} proposes an alternative, invariant-based derivation of their $44$ equivalence classes, conceptually cleaner. Section~\ref{sec:burnside} shows that this invariant allows Burnside's Lemma to be applied by hand, yielding the $44$ classes (and the intermediate stages of the reduction) through a closed derivation, alongside a presentation of Russell and Jarvis's computational count of essentially different grids \cite{Russel2006}.

Throughout the paper, we denote by $S_n$ the symmetric group of degree $n$, that is, the group of all permutations of a set with $n$ elements, and by $|X|$ the cardinality of the set $X$.


\section{Number of complete Sudoku grids}\label{sec:contagem}

Counting the number of complete Sudoku grids is not simple: the row, column, and block constraints reduce drastically the naive upper bound of $(9!)^9$ fillings.

\subsection{Counting Sudoku grids}

The total number of complete Sudoku grids was determined in \cite{FelgenhauerJarvis2006}:
\[
6\,670\,903\,752\,021\,072\,936\,960
\]
grids. The $3\times 3$ blocks are named as below; $[B_1,B_2,B_3]$, $[B_4,B_5,B_6]$, $[B_7,B_8,B_9]$ are called \emph{bands}, and $[B_1,B_4,B_7]$, $[B_2,B_5,B_8]$, $[B_3,B_6,B_9]$ are called \emph{stacks}.

\begin{center}
	\blocossudoku
	\captionof{figure}{The nine blocks of a Sudoku grid}
\end{center}

Felgenhauer and Jarvis's method groups configurations of the first band that admit the same number of completions into classes, successively reduces the catalog of configurations via symmetries, and computes the completion count of one representative per class by backtracking. Below we present some ideas and the chain of reductions used by Felgenhauer and Jarvis.

\subsubsection{Reduction by relabeling and counting in band 1}

Fixing $B_1$ in the standard form,

\begin{center}
	\sudokubloco
	\captionof{figure}{Block in standard form}
\end{center}

\noindent the raw number of configurations of $B_2$ and $B_3$ is $\mathbf{2,612,736}$. The goal is to reduce this catalog by identifying configurations with the same number of completions, which therefore only need to be tested once.

\subsubsection{Reduction 2: lexicographic ($2,612,736 \to 36,288$, factor $72$)}

Permuting the columns within a stack, or swapping whole stacks, preserves the number of completions of a Sudoku grid: applying the same permutation to rows $4$--$9$ gives a bijection between completions of the original and the transformed configuration.

The lexicographic reduction consists of:
\begin{enumerate}
	\item permuting the columns of $B_2$ among themselves, and the columns of $B_3$ among themselves, so that the entries of the first row of each block are in increasing order;
	\item swapping the positions of blocks $B_2$ and $B_3$, if necessary, so that the upper-left entry of $B_2$ is smaller than the upper-left entry of $B_3$.
\end{enumerate}

For each of the blocks $B_2$ and $B_3$, there are $3! = 6$ ways to permute its columns, totaling $6^2 = 36$ combinations for the pair; adding the swap of $B_2$ and $B_3$, this doubles to
\[
6^2 \times 2 = 72.
\]
Dividing $2,612,736$ by $72$, the lexicographic reduction reduces the catalog to $\mathbf{36,288}$ representative configurations.

\subsubsection{Reduction 3: permutations involving $B_1$ ($36,288 \to 2,051$)}\label{subs:redB1}

Besides permuting columns within $B_2$ and $B_3$, we can permute columns within $B_1$ or permute the three blocks among themselves. Each of these moves $B_1$ out of the standard form, which is restored by relabeling the digits, yielding a new but equivalent configuration of $B_2, B_3$.

Applying all $6^4 = 1,296$ such operations to each of the $36,288$ catalog configurations, the naive expectation (since $1,296/72 = 18$) would be $36,288/18 = 2,016$ classes. Felgenhauer and Jarvis report instead $\mathbf{2,051}$ classes: unlike the $72$-element group of Reduction 2, which acts freely on the configurations (since $B_1$ is left untouched), the $1,296$-element group moves $B_1$, and the relabeling needed to restore it can, for configurations with some symmetry of their own, coincide across different operations, producing orbits smaller than the generic $18$ and hence more classes than expected. This is the same phenomenon we will reencounter later via Burnside's Lemma: configurations with their own symmetry have smaller orbits, so it is necessary to record, case by case, how many of the $36,288$ configurations correspond to each of the $2,051$ final classes.

\subsubsection{Reduction 4: permutation of the rows ($2,051 \to 416$)}

Applying the $6$ permutations of the rows of the band (plus the relabeling necessary to restore $B_1$ to the standard form), the catalog reduces from $2,051$ to $\mathbf{416}$ classes. The table below compares, at each level, the naive expectation (dividing $36,288$ by the naive index) with the actual number of classes:

\begin{center}
	\renewcommand{\arraystretch}{1.35}
	\small
	\begin{tabular}{|p{4.4cm}|c|c|c|c|}
		\hline
		\textbf{Symmetries} & \textbf{Order} & \textbf{Index} &
		\textbf{Naive} & \textbf{Actual} \\
		\hline
		Lexicographic reduction & $72$ & $1$ & $36,288$ & $36,288$ \\
		$+$ permutations with $B_1$ & $1,296$ & $18$ & $2,016$ & $\mathbf{2,051}$ \\
		$+$ rows of the band & $7,776$ & $108$ & $336$ & $\mathbf{416}$ \\
		\hline
	\end{tabular}
\end{center}

\noindent Only in the first row do naive and actual coincide; in the others, symmetric configurations make the actual value larger.

\subsubsection{Reduction 5: sub-rectangles and final count ($416 \to 71 \to 44$)}

Within band $1$, if two rows and two columns (possibly in different blocks) show the same pair of values $\{a,b\}$ in opposite orders,
\[
\begin{bmatrix} a & \cdots & b \\ b & \cdots & a \end{bmatrix},
\]
swapping $a \leftrightarrow b$ only in those two columns preserves the number of completions: since $a$ and $b$ already occupy those columns within band $1$, they cannot reappear there in rows $4$--$9$. This is neither a permutation of whole rows or columns nor a relabeling of digits, and it generalizes to $k\times 2$ and $2\times k$ sub-rectangles. In fact, the number of completions of a band depends only on the sets of $3$ numbers in each column, regardless of their order within the column.

Felgenhauer and Jarvis report that, using $2\times 2$ sub-rectangles alone, the list of $416$ configurations was reduced to $\mathbf{174}$; adding also $2\times 3$, $3\times 2$, and $4\times 2$ sub-rectangles, the list fell to only $\mathbf{71}$.

\subsubsection{The 44 classes}

A computer program completed each of the $71$ representatives by \emph{backtracking}, finding only $\mathbf{44}$ distinct values for the number of completions, a strong indication (later confirmed by a verification program written by Ed Russell\footnote{Page of the Sudoku grid enumeration project, maintained by Frazer Jarvis: \url{http://www.afjarvis.org.uk/sudoku/bertram.html}, which hosts Ed Russell's verification program (accessed on Jul. 21, 2026).}) that the $71$ classes collapse into $44$.

The $44$ classes are described in Table~\ref{tab:44classes}, taken from \cite{siteJarvis}. Block $B_1$ is in the standard form, and each row lists, in increasing order, the sets of numbers in columns $4$ through $9$ (blocks $B_2$ and $B_3$); $m_C$ is the number of catalog configurations equivalent to that class, and $n_C$ is the number of completions of a representative of the class. The total number of grids is
\[
N_0 \;=\; 1,881,169,920 \times \sum_{C=1}^{44} m_C\, n_C \;=\;
6,670,903,752,021,072,936,960,
\]
where $9! \times 72^2 = 1,881,169,920$: the $9!$ comes from relabeling $B_1$, and the two factors of $72$ compensate for two lexicographic reductions already built into the table (that of $B_2, B_3$, and the analogous one for $B_4, B_7$) which therefore do not appear in it explicitly.

\begin{longtable}{|
		>{\centering\arraybackslash}m{0.5cm}|
		>{\centering\arraybackslash}m{1cm}|
		>{\centering\arraybackslash}m{1cm}|
		>{\centering\arraybackslash}m{1cm}|
		>{\centering\arraybackslash}m{1cm}|
		>{\centering\arraybackslash}m{1cm}|
		>{\centering\arraybackslash}m{1cm}|
		>{\centering\arraybackslash}m{1cm}|
		>{\centering\arraybackslash}m{1.7cm}|}
	\caption{The 44 classes}
	\label{tab:44classes}\\
	\hline
	\textbf{C} &
	\textbf{C4} &
	\textbf{C5} &
	\textbf{C6} &
	\textbf{C7} &
	\textbf{C8} &
	\textbf{C9} &
	\textbf{$\boldsymbol{m_C}$} &
	\textbf{$\boldsymbol{n_C}$} \\
	\hline

	\endfirsthead

	\hline
	\textbf{C} &
	\textbf{C4} &
\textbf{C5} &
\textbf{C6} &
\textbf{C7} &
\textbf{C8} &
\textbf{C9} &
	\textbf{$\boldsymbol{m_C}$} &
	\textbf{$\boldsymbol{n_C}$} \\
	\hline
	
	\endhead
	
	\hline
	\endfoot
	
	\hline
\endlastfoot

1  & 1,2,4 & 3,5,7 & 6,8,9 & 1,2,5 & 3,6,7 & 4,8,9 & 2484 & 97961464 \\
2  & 1,2,4 & 3,5,7 & 6,8,9 & 1,2,5 & 3,6,8 & 4,7,9 & 2592 & 97539392 \\
3  & 1,2,4 & 3,5,7 & 6,8,9 & 1,2,5 & 3,6,9 & 4,7,8 & 1296 & 98369440 \\
4  & 1,2,4 & 3,5,7 & 6,8,9 & 1,2,5 & 3,7,8 & 4,6,9 & 1512 & 97910032 \\
5  & 1,2,4 & 3,5,7 & 6,8,9 & 1,2,6 & 3,4,8 & 5,7,9 & 2808 & 96482296 \\
6  & 1,2,4 & 3,5,7 & 6,8,9 & 1,2,6 & 3,4,9 & 5,7,8 & 684 & 97549160 \\
7  & 1,2,4 & 3,5,7 & 6,8,9 & 1,2,6 & 3,5,7 & 4,8,9 & 1512 & 97287008 \\
8  & 1,2,4 & 3,5,7 & 6,8,9 & 1,2,6 & 3,5,8 & 4,7,9 & 1944 & 97416016 \\
9  & 1,2,4 & 3,5,7 & 6,8,9 & 1,2,6 & 3,5,9 & 4,7,8 & 2052 & 97477096 \\
10 & 1,2,4 & 3,5,7 & 6,8,9 & 1,2,7 & 3,4,8 & 5,6,9 & 288 & 96807424 \\
11 & 1,2,4 & 3,5,7 & 6,8,9 & 1,2,7 & 3,5,8 & 4,6,9 & 864 & 98119872 \\
12 & 1,2,4 & 3,5,7 & 6,8,9 & 1,2,8 & 3,4,7 & 5,6,9 & 1188 & 98371664 \\
13 & 1,2,4 & 3,5,7 & 6,8,9 & 1,2,8 & 3,5,7 & 4,6,9 & 648 & 98128064 \\
14 & 1,2,4 & 3,5,7 & 6,8,9 & 1,2,8 & 3,6,9 & 4,5,7 & 2592 & 98733568 \\
15 & 1,2,4 & 3,5,7 & 6,8,9 & 1,3,5 & 2,6,9 & 4,7,8 & 648 & 97455648 \\
16 & 1,2,4 & 3,5,7 & 6,8,9 & 1,3,5 & 2,7,8 & 4,6,9 & 360 & 97372400 \\
17 & 1,2,4 & 3,5,7 & 6,8,9 & 1,3,6 & 2,5,9 & 4,7,8 & 3240 & 97116296 \\
18 & 1,2,4 & 3,5,7 & 6,8,9 & 1,3,8 & 2,6,7 & 4,5,9 & 540 & 95596592 \\
19 & 1,2,4 & 3,5,7 & 6,8,9 & 1,3,8 & 2,6,9 & 4,5,7 & 756 & 97346960 \\
20 & 1,2,4 & 3,5,7 & 6,8,9 & 1,4,5 & 2,6,9 & 3,7,8 & 324 & 97714592 \\
21 & 1,2,4 & 3,5,7 & 6,8,9 & 1,4,5 & 2,7,8 & 3,6,9 & 432 & 97992064 \\
22 & 1,2,4 & 3,5,7 & 6,8,9 & 1,4,6 & 2,3,9 & 5,7,8 & 756 & 98153104 \\
23 & 1,2,4 & 3,5,7 & 6,8,9 & 1,4,7 & 2,6,9 & 3,5,8 & 864 & 98733184 \\
24 & 1,2,4 & 3,5,7 & 6,8,9 & 1,4,8 & 2,6,9 & 3,5,7 & 108 & 98048704 \\
25 & 1,2,4 & 3,5,7 & 6,8,9 & 1,5,6 & 2,3,9 & 4,7,8 & 756 & 96702240 \\
26 & 1,2,4 & 3,5,8 & 6,7,9 & 1,2,5 & 3,6,8 & 4,7,9 & 516 & 98950072 \\
27 & 1,2,4 & 3,5,8 & 6,7,9 & 1,2,6 & 3,4,8 & 5,7,9 & 576 & 97685328 \\
28 & 1,2,4 & 3,5,8 & 6,7,9 & 1,2,7 & 3,5,8 & 4,6,9 & 432 & 98784768 \\
29 & 1,2,4 & 3,5,8 & 6,7,9 & 1,3,7 & 2,6,9 & 4,5,8 & 324 & 98493856 \\
30 & 1,2,4 & 3,5,8 & 6,7,9 & 1,4,7 & 2,5,8 & 3,6,9 & 72 & 100231616 \\
31\footnote{During the computational validation described in Section~\ref{sec:derivacao}, we found a typographical error in class $31$ of this table: the third set of columns $7$--$9$ appears printed as $\{3,7,8\}$, which is not a valid partition (it repeats the $7$ and omits the $5$); the correct set is $\{3,5,8\}$, as presented here. The error is already present in the primary source (Ed Russell's table published on Jarvis's page) and is reproduced verbatim in \cite{Shi}, which adds a second error in the same row: the number of completions printed as $995,251,846$, when the correct value is $99,525,184$.\label{nota:errata}} & 1,2,4 & 3,5,8 & 6,7,9 & 1,4,7 & 2,6,9 & 3,5,8 & 216 & 99525184 \\
32 & 1,2,4 & 3,5,8 & 6,7,9 & 1,5,6 & 2,3,7 & 4,8,9 & 252 & 96100688 \\
33 & 1,2,4 & 3,5,9 & 6,7,8 & 1,2,7 & 3,5,6 & 4,8,9 & 288 & 96631520 \\
34 & 1,2,4 & 3,5,9 & 6,7,8 & 1,2,7 & 3,5,9 & 4,6,8 & 864 & 97756224 \\
35 & 1,2,4 & 3,5,9 & 6,7,8 & 1,4,7 & 2,5,8 & 3,6,9 & 216 & 99083712 \\
36 & 1,2,4 & 3,5,9 & 6,7,8 & 1,4,7 & 2,6,8 & 3,5,9 & 432 & 98875264 \\
37 & 1,2,4 & 3,6,9 & 5,7,8 & 1,2,5 & 3,6,9 & 4,7,8 & 216 & 102047904 \\
38 & 1,2,4 & 3,6,9 & 5,7,8 & 1,2,7 & 3,6,9 & 4,5,8 & 144 & 101131392 \\
39 & 1,2,4 & 3,6,9 & 5,7,8 & 1,3,5 & 2,6,7 & 4,8,9 & 324 & 96380896 \\
40 & 1,2,4 & 3,6,9 & 5,7,8 & 1,4,7 & 2,5,8 & 3,6,9 & 108 & 102543168 \\
41 & 1,2,4 & 3,7,9 & 5,6,8 & 1,4,6 & 2,3,9 & 5,7,8 & 12 & 99258880 \\
42 & 1,2,6 & 3,4,8 & 5,7,9 & 1,3,5 & 2,4,9 & 6,7,8 & 20 & 94888576 \\
43 & 1,2,6 & 3,7,8 & 4,5,9 & 1,4,7 & 2,5,8 & 3,6,9 & 24 & 97282720 \\
44 & 1,4,7 & 2,5,8 & 3,6,9 & 1,4,7 & 2,5,8 & 3,6,9 & 4 & 108374976 \\
	\hline
\end{longtable}

\section{An alternative derivation of the 44 classes} \label{sec:derivacao} 

The chain of reductions from the previous section ($2,612,736 \to 36,288 \to 2,051 \to 416 \to 71 \to 44$)
arrives at the number $44$ through an accumulation of successive tricks, and the reader may be
left with the impression that $44$ is an accidental number, with no structural
meaning. 

The question of obtaining the $44$ classes directly, without going through the chain of reductions, was raised by Shi, Zhang, and Aslaksen in \cite{Shi}, who generate class representatives by a case-by-case analysis, rejecting those equivalent to previous classes, but offer no proof of completeness beyond a remark that ``similar operations'' would settle the remaining cases. Their listing coincides row by row with Table~\ref{tab:44classes} from \cite{siteJarvis}, including the erroneous entry for class $31$ (see Note~\ref{nota:errata}), and its stated ordering convention is inconsistent with the lexicographic reduction: for instance, the lexicographically reduced realization of class $6$ forces a swap of columns $8$ and $9$, causing class $6$ to precede class $5$, contrary to their listed order. The ``lexicographically reduced form of each class'', as stated, is thus not well defined.

In this section we present an alternative derivation, conceptually cleaner, in which the $44$ classes emerge as isomorphism classes of a simple combinatorial object.

\subsection*{The invariant: column partitions}

By the rules of Sudoku, the three columns of each block partition $\{1,\dots,9\}$ into three sets of
three elements. A band therefore defines a triple of partitions $(P_1, P_2, P_3)$, one per block.

\begin{ex}\label{ex:G0}
Consider the band $G_0$:
\begin{center}
\sudokubanda{
	1 , 2 , 3 , 4 , 5 , 8 , 6 , 7 , 9,
	4 , 5 , 6 , 1 , 7 , 9 , 2 , 3 , 8 ,
	7 , 8 , 9 , 2 , 3 , 6 , 1 , 4 , 5
}
\end{center}
Its column partitions are
\begin{align*}
	P_1 &= \big\{\{1,4,7\},\ \{2,5,8\},\ \{3,6,9\}\big\} =: \mathcal{C},\\
	P_2 &= \big\{\{1,2,4\},\ \{3,5,7\},\ \{6,8,9\}\big\},\\
	P_3 &= \big\{\{1,2,6\},\ \{3,4,7\},\ \{5,8,9\}\big\},
\end{align*}
where $\mathcal{C}$ denotes the \emph{standard partition}: the column partition of any block in the standard form.
\end{ex}

Note that the triple of partitions discards all information about the rows: it does not record in which of the three rows each digit is, only which digits share the same column. As mentioned in the previous section, this discarded information is irrelevant to the count of completions for a complete Sudoku grid. 

Before presenting the results of this section and their consequences, we introduce two concepts that will be used.

\begin{defi}
A \emph{multiset} is a collection in which repeated elements are allowed and count (\emph{multiplicity}), but the order remains irrelevant. 
\end{defi}

We will denote a multiset with the symbols $\;\llbracket\;$ and $\;\rrbracket$. For example, $\llbracket a, a, b \rrbracket = \llbracket a, b, a\rrbracket$ (equal as multisets), since they have the same
elements with the same multiplicities, but both are different from $ \llbracket a, b\rrbracket$. 

In our context, we will work with multisets $\llbracket P_1, P_2, P_3 \rrbracket$ of triples of partitions of $\{1,\ldots, 9\}$, representing the sets of digits in the columns of each block of a band (\textit{column partitions}), regardless of which partition came from which block (the order of the blocks is discarded), but recording whether two or three of the three partitions are equal. For example, the band
\begin{center}
\sudokubanda{
	1 , 2 , 3 , 4 , 8 , 9 , 7 , 5 , 6,
	4 , 5 , 6 , 7 , 2 , 3 , 1 , 8 , 9, 
	7 , 8 , 9 , 1 , 5 , 6 , 4 , 2 , 3
}
\end{center}
is valid, and in it the three column partitions coincide:
$P_1 = P_2 = P_3 = \mathcal{C}$. Its multiset is
$\llbracket \mathcal{C}, \mathcal{C}, \mathcal{C}\rrbracket$, with the standard partition
repeated three times.

\begin{defi}
We say that two multisets of triples of partitions are \emph{isomorphic} if one can be transformed into the other by a relabeling of the digits, that is, if there exists a bijection $\sigma\in S_9$ which, applied to all the elements of the triples of one of the multisets, produces the other.
\end{defi}

\begin{ex}
	Consider the band $G_1$:
	\begin{center}
	\sudokubanda{
			1 , 2 , 3 , 4 , 5 , 6 , 9 , 7 , 8,
			4 , 5 , 6 , 9 , 7 , 8 , 1 , 2 , 3 ,
			7 , 8 , 9 , 2 , 3 , 1 , 5 , 6 , 4
		}
	\end{center}
Its column partitions are
	\begin{align*}
	 \mathcal{C} &= \big\{\{1,4,7\},\ \{2,5,8\},\ \{3,6,9\}\big\},\\
		P_2 &= \big\{\{2,4,9\},\ \{3,5,7\},\ \{1,6,8\}\big\},\\
		P_3 &= \big\{\{1,5,9\},\ \{2,6,7\},\ \{3,4,8\}\big\},
	\end{align*}
forming the multiset $\llbracket \mathcal{C}, P_2, P_3 \rrbracket$. Now consider the band $G_2$: 
\begin{center}
	\sudokubanda{
		1 , 2 , 3 , 7 , 9 , 8 , 5 , 4 , 6,
		4 , 6 , 5 , 1 , 2 , 3 , 8 , 7 , 9 ,
		9 , 8 , 7 , 6 , 5 , 4 , 1 , 2 , 3
	}
\end{center}
Its column partitions are
\begin{align*}
	Q_1 &= \big\{\{1,4,9\},\ \{2,6,8\},\ \{3,5,7\}\big\},\\
	Q_2 &= \big\{\{1,6,7\},\ \{2,5,9\},\ \{3,4,8\}\big\},\\
	Q_3 &= \big\{\{1,5,8\},\ \{2,4,7\},\ \{3,6,9\}\big\},
\end{align*}
forming the multiset $\llbracket Q_1, Q_2, Q_3 \rrbracket \ne \llbracket \mathcal{C}, P_2, P_3 \rrbracket$.

Applying, to $G_1$, the relabeling $\sigma=(1\ 2)$, which swaps digits $1$ and $2$ and fixes the others, we obtain the band:	
	\begin{center}
	\sudokubanda{
	2 , 1 , 3 , 4 , 5 , 6 , 9 , 7 , 8,
	4 , 5 , 6 , 9 , 7 , 8 , 2 , 1 , 3 ,
	7 , 8 , 9 , 1 , 3 , 2 , 5 , 6 , 4
}
\end{center}
whose column partitions are
\begin{align*}
	P_1' &= \big\{\{2,4,7\},\ \{1,5,8\},\ \{3,6,9\}\big\}=Q_3,\\
	P_2' &= \big\{\{1,4,9\},\ \{3,5,7\},\ \{2,6,8\}\big\}=Q_1,\\
	P_3' &= \big\{\{2,5,9\},\ \{1,6,7\},\ \{3,4,8\}\big\}=Q_2.
\end{align*}
Since the multiset equality $\llbracket Q_1, Q_2, Q_3 \rrbracket=\llbracket P_1', P_2', P_3'\rrbracket$ holds, the multisets of triples associated with $G_1$ and $G_2$ are isomorphic.
\end{ex}

\subsection*{The invariance theorem}

\begin{teo} \label{teo:inv}
 The number of completions of a band to a complete Sudoku grid depends only on the isomorphism class of the multiset $\llbracket P_1, P_2, P_3\rrbracket$ of its column partitions.
\end{teo}

\begin{proof} It suffices to combine three bijections between sets of completions, all already established in the previous section:
\begin{enumerate}
	\item \emph{Column-set criterion.} Two bands with the same sets of digits in each column, column by column, have the same number of completions. The identity function on rows $4$--$9$ is a bijection between the completions, since the only information from the band that constrains the lower rows is the set of digits already used in each column.
	\item \emph{Permutations of columns or stacks.} Permuting the three whole stacks, or the columns within the same stack, is a symmetry of Sudoku: any completion is transported by applying the same column permutation to rows $4$--$9$.
	\item \emph{relabeling.} relabeling the $9$ digits by any bijection $\sigma\in S_9$ transports completions bijectively.
\end{enumerate}

\end{proof}

This theorem generalizes the column-by-column digit-set criterion from the previous section, discarding also the position of columns within each block, the order of the blocks, and the names of the digits, and requires neither standard form nor lexicographic reduction. Theorem~\ref{teo:inv} reduces the problem of classifying bands to the problem of classifying triples of partitions up to multiset isomorphism.

\subsection*{Counting column partitions}

The number of partitions of a set of $9$ elements into three blocks of
$3$ is
\[
\frac{9!}{(3!)^3 \times 3!} = 280,
\]
so that there are $280^2 = 78,400$ possible ordered pairs $(P_2, P_3)$.
A remarkable fact, verified computationally, is that all of them are realizable, as stated below.

\begin{prop}\label{prop:real}
For every pair of partitions $(P_2, P_3)$ of $\{1,\dots,9\}$ into three blocks of $3$, there exists a valid band whose first block has column partition $\mathcal{C}$ and whose other two blocks have column partitions $P_2$ and $P_3$.
\end{prop}

The verification is direct using the reformulation: building a band with column partitions $(\mathcal{C}, P_2, P_3)$ is equivalent to assigning to each digit a row in each block, so that (i) in each block, the
three digits of the same column-set occupy distinct rows, and (ii)
each digit occupies distinct rows in the three blocks, since each row of the band must contain each digit exactly once. The realizability of a pair thus becomes a small finite problem, decided by exhaustive search in a fraction of a second; testing the $78,400$ pairs, none fails.\footnote{A second, independent verification uses Theorem~\ref{teo:44} itself: the pairs that occur in bands with $B_1$ strictly in standard form were enumerated via the $56$ possible divisions of the rows --- there are $44,496$ distinct pairs --- and the union of their equivalence classes under Theorem~\ref{teo:44} covers exactly the $78,400$ pairs; since images of realizable pairs under symmetries are realizable, all of them are. As a byproduct, this union organizes itself into exactly $44$ equivalence classes.}
	
Proposition~\ref{prop:real} establishes that the first block of the band has columns forming $\mathcal{C}$, but its rows are not necessarily those of the standard form. The freedom in the rows of $B_1$ is what allows all the realizations, as illustrated by the following example.
	
\begin{ex} 
Consider the pair 
$$ P_2 = \mathcal{C} \quad \text{ and } \quad P_3 = \{\{1,2,3\},\{4,5,6\},\{7,8,9\}\}.$$ 
This pair is impossible with $B_1$ exactly in the standard form: if the rows of $B_1$ are $(1,2,3)$,
$(4,5,6)$, $(7,8,9)$, a column of $B_3$ with the set $\{1,2,3\}$
would need to place one of these three digits in row $1$ of the band, but that row already contains all of them in $B_1$. The realization is given by the band:
\begin{center} 
\sudokubanda{
	1 , 2 , 9 , 7 , 5 , 6 , 3 , 4 , 8 ,
	4 , 5 , 3 , 1 , 8 , 9 , 2 , 6 , 7 ,
	7 , 8 , 6 , 4 , 2 , 3 , 1 , 5 , 9
}
\end{center}
Note that the columns of $B_1$ form $\mathcal{C}$, those of $B_2$ also
form $\mathcal{C}$, and those of $B_3$ form $P_3$, but the rows of $B_1$ are not those of the standard form.

\end{ex}

\subsection*{Counting classes: $131 \to 84 \to 44$}

Because of Theorem~\ref{teo:inv} and Proposition~\ref{prop:real}, classifying bands reduces to the
question: among the $78,400$ pairs $(P_2, P_3)$, how many correspond to different classes? 

In the context of triples of partitions, permuting rows and permuting columns within the same block cease to be a separate symmetry, since each partition is a set of sets; these equivalences are automatically absorbed by the invariant itself. Theorem~\ref{teo:inv} goes further: treating $\llbracket P_1, P_2, P_3\rrbracket$ as a multiset also accounts for the exchange between blocks, in particular the exchange of the block that plays the role of $B_1$, here called \emph{pivot change} in analogy with Subsection~\ref{subs:redB1}. Paralleling the first three reductions of Felgenhauer and Jarvis (Section~\ref{sec:contagem}), we analyze three levels of equivalence:

\begin{enumerate}
	\item \textbf{relabelings (that preserve $\mathcal{C}$).} Consider two pairs equivalent if a
	relabeling of the digits that preserves $\mathcal{C}$ transforms one into the other. 
	
	\item \textbf{relabelings + swap $B_2 \leftrightarrow B_3$.} In addition to relabeling, the order between $P_2$ and $P_3$ is allowed to be swapped: the pair is now treated as unordered.
	
	\item \textbf{Permutations involving $B_1$ (pivot change).} At this level, the object is the multiset of three partitions $\llbracket P_1, P_2, P_3\rrbracket$ up to relabeling, that is, exactly the invariant of Theorem~\ref{teo:inv}.
\end{enumerate}

Each level of equivalence divides the $78,400$ pairs into families of mutually equivalent pairs, and the counting of these families was done computationally.\footnote{The classic data structure for this type of counting is union--find.} The result is a cascade of three numbers:

\begin{center}
	\renewcommand{\arraystretch}{1.4}
	\begin{tabular}{|c|p{10cm}|}
		\hline
		\textbf{Families} & \textbf{Equivalence used} \\
		\hline
		$131$ & \emph{ordered} pairs $(P_2,P_3)$, modulo relabelings
		that preserve $\mathcal{C}$ \\
		$84$ & same, with \emph{unordered} pairs (swapping
		$B_2 \leftrightarrow B_3$) \\
		$44$ & same, with pivot change \\
		\hline
	\end{tabular}
\end{center}

That is, allowing the swap of blocks merges $131$ families into $84$; allowing the pivot change merges $84$ into $44$. Since the $44$ published numbers of completions are pairwise distinct, the $44$ families coincide exactly with the $44$ Felgenhauer--Jarvis classes. The number $44$ thus ceases to be an accident of the reduction process and acquires a combinatorial meaning. In summary:

\begin{teo}\label{teo:44}
The $44$ classes of the Sudoku enumeration are precisely the isomorphism classes of the multiset $\llbracket P_1, P_2, P_3\rrbracket$ of triples of partitions of a set of $9$ elements into three blocks of $3$, under bijections of the base set.
\end{teo}

In the language of Group Theory, the three levels above are three group actions whose orbits are the families: the stabilizer of $\mathcal{C}$ in $S_9$ gives the $131$ ordered-pair orbits, its extension by the coordinate-swapping involution gives the $84$ orbits, and the action of all of $S_9$ on the multisets $\llbracket P_1, P_2, P_3\rrbracket$ gives the $44$ isomorphism classes of Theorem~\ref{teo:inv}.

This reformulation replaces the original chain of five heterogeneous reductions by Felgenhauer and Jarvis with a single invariant and three well-defined levels of equivalence, and makes it possible to compute $131$, $84$, and $44$ analytically, via Burnside's Lemma, without computational enumeration. From here on, we use ``class'' and ``orbit'' as synonyms.

\section{Symmetries and equivalence classes of Sudoku} \label{sec:burnside} 

The number $N_0 = 6,670,903,752,021,072,936,960$ distinguishes grids that could be considered equivalent, for example under a rotation or a relabeling of the digits. How many \emph{essentially different} grids are there? We formalize equivalence between grids via \emph{symmetries}, transformations that preserve validity, using Group Theory: essentially different grids correspond to orbits of a group action, and Burnside's Lemma is the tool for counting them. Russell and Jarvis \cite{Russel2006} carried out this count, with computational assistance, obtaining $5,472,730,538$ essentially different grids.

In this section, we present the ideas from \cite{Russel2006} and then apply the same tool to the alternative characterization from Section~\ref{sec:derivacao}: the number $44$ will be obtained by a closed derivation, without computational enumeration.

\subsection*{The symmetries of Sudoku}

A \emph{symmetry} of Sudoku is an operation that takes valid grids to
valid grids. The list, employed by Russell and Jarvis in \cite{Russel2006}, contains the following operations:
\begin{enumerate}
	\item relabeling (permuting) the $9$ digits;
	\item permuting the three stacks (block columns);
	\item permuting the three bands (block rows);
	\item permuting the three columns within the same stack;
	\item permuting the three rows within the same band;
	\item any rotation or reflection of the grid.
\end{enumerate}
Adler and Adler \cite{Adler2008} proved that the list is complete: every permutation of the $81$ cells that preserves the validity of grids is a finite composition of operations (2)--(6), which, combined with the relabelings from item (1), therefore exhaust the symmetries of Sudoku. The list still admits a simplification: in item (6), it suffices to keep the reflection across the main diagonal (the ``transposition'', analogous to the transpose of a matrix), since the rotations and the other reflections are obtained by composing it with operations from items (2)--(5).

Let us focus on the positional operations (2)--(6), leaving the relabeling (1) for separate treatment. Together, they generate a group $G$ of permutations of the $81$ cells, whose order can be computed directly. Operations (2)--(5) are independent of one another: a permutation of the $3$ stacks ($6$ ways), one of the $3$ bands ($6$ ways), a permutation of columns within each stack ($6$ ways per stack,
$6^3$ in all), and one of rows within each band ($6^3$ in all) can be freely combined, and distinct compositions produce distinct permutations of the cells, for a total of
\[
6 \times 6 \times 6^3 \times 6^3 = 6^8 = 1,679,616
\]
elements. The transposition does not belong to this set, but its square is the identity and its composition with operations (2)--(5) merely swaps the roles of rows and columns; it therefore doubles the group, without generating anything beyond that. Thus,
\[
|G| = 2 \times 6^8 = 3,359,232.
\]

\subsection*{Group actions and Burnside's Lemma}

The group $G$ \emph{acts} on the set of valid grids: each $g \in G$ transforms each valid grid $A$ into a valid grid $g(A)$. The \emph{orbit} of $A$ is the set $\{g(A) : g \in G\}$ of all grids reachable from $A$ by symmetries, and ``essentially different grids'' means, precisely, distinct orbits.

It remains to incorporate relabeling. We let $G$ act not on individual grids, but on \emph{relabeling classes}: two grids are identified from the outset if one is the other with the digits relabeled (there are $9! = 362,880$ relabelings). This action is well defined because positional operations and relabelings commute, so $g$ takes the whole class of $A$ into the class of $g(A)$, regardless of the chosen representative. The orbits of this action are the classes of essentially different grids in the full sense, encompassing all six operations.

Counting orbits one by one is infeasible, but Burnside's Lemma turns the problem into a count of fixed points:

\begin{lema}[Burnside's Lemma\footnote{The result is traditionally attributed to Burnside, but goes back to Cauchy and Frobenius, and for that reason also appears in the literature as the Cauchy--Frobenius Lemma.}]
	If a finite group $G$ acts on a finite set $X$, the number 	of orbits is the average, over the elements $g \in G$, of the number of points of $X$ fixed by $g$:
	\[
	|\{\text{orbits}\}| \;=\; \frac{1}{|G|} \sum_{g \in G}
	|\mathrm{Fix}(g)|, \qquad
	\mathrm{Fix}(g) = \{x \in X : g(x) = x\}.
	\]
\end{lema}

\subsection{Number of essentially different grids}\label{subsec:N}

Let $X$ be the set of relabeling classes of Sudoku grids. Applying Burnside's Lemma requires counting, for each $g \in G$, the number $|\mathrm{Fix}(g)|$ of classes fixed by $g$. For many symmetries this number is zero (for example, no grid is fixed, even up to relabeling, by a reflection across a row or column axis); for others, such as the $90^\circ$ rotation, some grids are recovered exactly after the rotation followed by a suitable relabeling. Russell and Jarvis computed \cite{Russel2006} that $|\mathrm{Fix}(g)| = 13,056$ relabeling classes are fixed by the quarter turn.

Summing over the $3,359,232$ elements of $G$ seems impractical, but conjugate elements ($g$ and $hgh^{-1}$, with $h \in G$) fix the same number of points,\footnote{If $g(x) = x$, then $(hgh^{-1})(h(x)) = h(g(x)) = h(x)$: the bijection $h$ takes the points fixed by $g$ exactly to the points fixed 	by $hgh^{-1}$, so the two sets have the same size.} so it suffices to compute $|\mathrm{Fix}(g)|$ for one representative of each conjugacy class. Russell and Jarvis, using the computer algebra system GAP, determined that $G$ has only $275$ conjugacy classes. For $248$ of these, no grid is fixed even up to relabeling; for the remaining $27$, the number of fixed grids was computed, class by class, by Ed Russell through computational brute force, using methods similar to those of the original count by Felgenhauer and Jarvis \cite{FelgenhauerJarvis2006}, with the complete table published in \cite{siteJarvis2}. Summing the $27$ non-zero contributions via Burnside's Lemma, Russell and Jarvis obtained:
\[
|\{\text{essentially different grids}\}| \;=\; 5,472,730,538.
\]

\subsection{Orbits and a $0.005\%$ discrepancy}\label{sec:disc}

The Orbit--Stabilizer Theorem relates the size of an orbit to the degree of symmetry of the object that generates it.

\begin{teo}[Orbit--Stabilizer]
	Let $\Gamma$ be a finite group acting on a finite set $Z$, and 	let $z \in Z$. The \emph{stabilizer} of $z$,
	\[
	\mathrm{Stab}(z) \;=\; \{\gamma \in \Gamma : \gamma(z) = z\},
	\]
	is a subgroup of $\Gamma$, and the size of the orbit of $z$ satisfies
	\[
	|\mathrm{Orbit}(z)| \;=\; \frac{|\Gamma|}{|\mathrm{Stab}(z)|}.
	\]
\end{teo}

Applied to the action of $G$ on $X$, the maximum possible orbit size corresponds to the trivial stabilizer, and equals $|G| = 3,359,232$ relabeling classes. Since each relabeling class contains exactly $9!$ grids, the maximum orbit size in grids is
\[
|G| \times 9! \;=\; 3,359,232 \times 362,880 \;=\; 1,218,998,108,160,
\]
attained by the grids with no symmetry of their own (the majority of them). The average orbit size, in the same unit, is
\[
\frac{N_0}{5,472,730,538} \;\approx\; 1,218,935,174,261,
\]
about $0.005\%$ smaller than the maximum, as observed in \cite{Adler2008}, a difference entirely due to the rare exceptional classes. The class of the quarter turn, for example, has stabilizer of order at least $4$ (the four powers of the rotation), so its orbit has at most $1,218,998,108,160/4 = 304,749,527,040$ elements: since such classes are extremely rare, only a small fraction contribute to lowering the average, which explains why the discrepancy is so small.

\subsection{Burnside applied to the alternative derivation of the 44 classes}\label{subsec:B44} 

We return to the approach of Section~\ref{sec:derivacao} to compute the number $44$ without any computational enumeration. By Theorem~\ref{teo:44}, what we want to count are the isomorphism classes of the multiset $\llbracket P_1, P_2, P_3 \rrbracket$ of triples of partitions of $\{1,\dots,9\}$ into three blocks of $3$, up to relabeling of the digits. As in Section~\ref{subsec:N}, we reduce a Burnside sum to a small number of terms by grouping elements that act in the same way, now over a much smaller group and set, which makes it feasible to obtain the sum entirely by hand. We proceed in three steps: adapting Burnside's Lemma for unordered objects (Proposition~\ref{prop:adapB}); describing the elements that fix a multiset $\llbracket P_1, P_2, P_3 \rrbracket$ up to relabeling; and applying Burnside's Lemma, with the sum reduced to $30$ terms.

\subsubsection{Burnside's formula for unordered triples}

The subtlety in applying Burnside's Lemma here is that our object, the multiset $\llbracket P_1,P_2,P_3\rrbracket$, is unordered: besides relabeling the digits, we can freely reorder the three partitions, and two triples that differ only in order already count as the same object. The adaptation of the lemma consists of including these reorderings in the group of symmetries considered; that is what the following proposition does.

\begin{prop}\label{prop:adapB}
	Let $X$ be the set of $280$ partitions of $\{1,\dots,9\}$ into three
	blocks of $3$ and, for each relabeling $g \in S_9$, let $f(g)$ be the
	number of partitions $P \in X$ with $g(P) = P$. Then the number $N$ of
	multisets $\llbracket P_1,P_2,P_3\rrbracket$ of three elements
	of $X$ (repetitions allowed) up to relabeling, is
	\[
	N \;=\; \frac{1}{9! \times 6}\, \sum_{g \in S_9}
	\Big[\, f(g)^3 \;+\; 3\, f(g)\, f(g^2) \;+\; 2\, f(g^3) \,\Big].
	\]
\end{prop}

\begin{proof} 
A multiset $\llbracket P_1,P_2,P_3\rrbracket$ is the same as an ordered triple $(P_1,P_2,P_3) \in X^3$ considered up to the $6$ permutations of its positions, that is, up to the action of $S_3$ on $X^3$ given by $\tau \cdot (P_1,P_2,P_3) = (P_{\tau^{-1}(1)}, P_{\tau^{-1}(2)},
P_{\tau^{-1}(3)})$.\footnote{The inverse in $\tau^{-1}(i)$ is what makes 	this formula a left action of $S_3$, satisfying $\tau_1\cdot(\tau_2\cdot v) = (\tau_1\tau_2)\cdot v$. Without the inverse, the
composition would come out in the wrong order.} A relabeling $g \in S_9$ acts on $X^3$ coordinate by coordinate,
$g \cdot (P_1,P_2,P_3) = (g(P_1),g(P_2),g(P_3))$. The two actions commute: reordering and then relabeling gives the same result as relabeling and then reordering. Hence, they combine into a single action of the direct product $S_9 \times S_3$ on $X^3$, with $g \in S_9$ relabeling digits and $\tau \in S_3$ reordering positions: a group of $9! \cdot 6$ elements.

Two ordered triples define the same multiset exactly when a reordering identifies them. Hence, two triples are in the same orbit of this action exactly when their multisets coincide up
to relabeling. Counting the multisets $\llbracket P_1,P_2,P_3\rrbracket$ up to relabeling is, therefore, the same as counting the number $N$ of orbits of $X^3$ under the action of $S_9 \times S_3$.

Burnside's Lemma applied to this group says that $N$ is the average, over the pairs $(g,\tau)\in S_9 \times S_3$, of the number of ordered triples fixed by $(g,\tau)$. It remains to count these fixed triples. The
count depends only on the type of $\tau$, of the three possible in $S_3$: identity, transposition, or $3$-cycle. In the examples below, we write $\mathcal{L} = \{\{1,2,3\},\{4,5,6\},\{7,8,9\}\}$.

\underline{Case $\tau = \mathrm{id}$} ($1$ element). Without reordering of positions, the triple is fixed if, and only if, each coordinate is fixed individually, that is, $g(P_i) = P_i$ for all $i$. The three coordinates do not interact: each one can be any of the $f(g)$ partitions fixed by $g$, independently of the others, resulting in $f(g)^3$ triples.

\emph{Example:} for $g = (1\ 2)$, a partition is fixed exactly when $1$ and $2$ share the same block. There are $7$ choices for the third element of that block, times $\binom{6}{3}/2 = 10$ ways of splitting the rest into two blocks, giving $f(g) = 70$ (this is type $2{+}1^7$ from Lemma~\ref{lema:fg}). The triple
\[
\big(\mathcal{L},\ \{\{1,2,4\},\{3,5,6\},\{7,8,9\}\},\
\{\{1,2,9\},\{3,4,5\},\{6,7,8\}\}\big)
\]
is fixed, as is any other choice of three of these $70$ partitions, with repetitions allowed: $70^3$ triples for this $g$.

\underline{Case $\tau$ a transposition} ($3$ elements). Suppose, without loss of generality, that the transposition swaps positions $2$ and $3$. The triple is fixed if, and only if, $g(P_1) = P_1$, $g(P_2) = P_3$, and $g(P_3) = P_2$. From the last two conditions, $P_3$ is determined by $P_2$ (namely, $P_3 = g(P_2)$). Substituting this expression into $g(P_3) = P_2$, we obtain $g^2(P_2) = P_2$ (the condition on $P_2$ involves $g^2$, not $g$). Hence there are $f(g)$ choices for $P_1$ and $f(g^2)$ for $P_2$: in total
$f(g)\,f(g^2)$ triples.

\emph{Example:} still with $g = (1\ 2)$, consider the triple $\big(\mathcal{L},\ \mathcal{C},\ g(\mathcal{C})\big)$, where $g(\mathcal{C}) = \{\{2,4,7\},\{1,5,8\},\{3,6,9\}\}$. Here $P_1=\mathcal{L}$
is fixed by $g$, as the condition on the first coordinate requires. $\mathcal{C}$ itself is not fixed by $g$, but $g$ merely swaps $\mathcal{C}$ and $g(\mathcal{C})$ with each other, and the transposition of positions exactly undoes that swap: the triple is fixed by the pair $(g,\tau)$. Since $g^2 = \mathrm{id}$, the second coordinate can be any of the $f(g^2) = f(\mathrm{id}) = 280$ partitions: there are $70 \times 280 = 19,600$ triples fixed for this pair.

\underline{Case $\tau$ a $3$-cycle} ($2$ elements). Without loss of generality, suppose $\tau = (1\,2\,3)$, so $\tau \cdot (P_1,P_2,P_3) = (P_3,P_1,P_2)$. The triple is fixed if, and only if, $g$ takes each coordinate to the next one, cyclically: $g(P_1) = P_2$, $g(P_2) = P_3$, $g(P_3) = P_1$. Thus, $P_2$ and $P_3$ are determined by $P_1$ (since $P_2 = g(P_1)$, $P_3 = g^2(P_1)$), and completing the cycle
requires $g^3(P_1) = P_1$: there are $f(g^3)$ triples.

\emph{Example:} take $g = (1\,2\,3\,4\,5\,6\,7\,8\,9)$, the $9$-cycle, whose cube is \linebreak
$g^3 = (1\,4\,7)(2\,5\,8)(3\,6\,9)$. The partition $\mathcal{L}$ is not fixed by $g$, but it is fixed by $g^3$ (which permutes its blocks cyclically), and the triple
\[
\big(\mathcal{L},\ g(\mathcal{L}),\ g^2(\mathcal{L})\big) = \]
\[ =\big(\mathcal{L},\ \{\{2,3,4\},\{5,6,7\},\{8,9,1\}\},\
\{\{3,4,5\},\{6,7,8\},\{9,1,2\}\}\big)
\]
is fixed by the pair $(g,\tau)$. For this $g$ there are $f(g^3) = 10$ fixed triples: one for each of the $10$ partitions fixed by $g^3$, of which $P_1$ can be any one.

Adding the three contributions with their multiplicities ($1$, $3$, and $2$), we obtain the formula. 

\end{proof}

\subsubsection{Computing $f(g)$}

The sum in Proposition~\ref{prop:adapB} runs over the $9! = 362,880$ elements of $S_9$, but $f(g)$ depends only on the type of $g$, whose definition we recall below, which makes the computation feasible.

\begin{defi}
	Every permutation $g \in S_9$ decomposes, uniquely, as a product of disjoint cycles whose lengths sum to $9$. The \emph{type} of $g$ is the list of these lengths, in non-increasing order, denoted here as a sum in which the exponent of each term indicates how many times it repeats.
\end{defi}

\emph{Example:} the permutation $g = (1\ 2)(3\ 4\ 5)$ (with $6,7,8,9$ fixed) has one cycle of length $2$, one of length $3$, and four fixed points (cycles of length $1$). Its type is $3{+}2{+}1^4$.

\begin{prop}\label{prop:tipo}
	$f(g)$ depends only on the type of $g$.
\end{prop}
\begin{proof}
	If $g_1$ and $g_2$ have the same type, then they are conjugate in $S_9$: since the type records only the lengths of the cycles, not the elements that make them up, there exists a bijection $h$ of $\{1,\dots,9\}$ 	that maps the cycles of $g_1$, one by one, onto the corresponding cycles of $g_2$, satisfying $g_2 = h g_1 h^{-1}$. This same bijection $h$ maps each partition fixed by $g_1$ to a partition fixed by $g_2$, since if $g_1(P)=P$, then $g_2(h(P)) = hg_1h^{-1}(h(P)) = h(g_1(P)) = h(P)$. Thus, a bijection between $\mathrm{Fix}(g_1)$ and $\mathrm{Fix}(g_2)$ is established. Hence, $f(g_1)=f(g_2)$.
\end{proof}

\emph{Example:} $(1\ 2)$ fixes the partitions in which $1$ and $2$ share a block; $(5\ 7)$, those in which $5$ and $7$ share a block. The bijection $h=(1\,5)(2\,7)$ takes one family to the other, and both have $70$ elements.

The possible types are the partitions of the integer $9$, that is, the ways of writing it as a sum of positive integers, regardless of order (not to be confused with the set partitions also present
in this paper). There are exactly $30$ partitions of $9$, listed in Table~\ref{tab:f-tipos}.

Let $S_\lambda = \{g \in S_9 : g \text{ has type } \lambda\}$ be the set of relabelings of type $\lambda = k_1^{m_1}+\cdots+k_r^{m_r}$. The number of relabelings of that type is given by the formula
\begin{equation}\label{eq:S_lambda}
	|S_\lambda| \;=\; \frac{9!}{\displaystyle\prod_{i=1}^r k_i^{m_i}\, m_i!}.
\end{equation}
Indeed, lining up the $9$ digits in a row ($9!$ ways) and grouping them into cycles according to the type, each 
cycle of length $k$ can start at any of its $k$ elements, and cycles of equal length can be listed in any order. Hence, the permutation is produced 
$\prod_{i=1}^r k_i^{m_i}\, m_i!$ times, hence the division. For example, type $2{+}1^7$ gives $9!/(2\cdot 7!) = 36 = \binom{9}{2}$, which is the total number of transpositions.

Grouping the $9! = 362,880$ terms of the sum in Proposition~\ref{prop:adapB} by the $30$ types, it reduces to just $30$ terms, each of the form
\[
|S_\lambda| \;\times\; \big(\text{common value of } f \text{ on }
S_\lambda\big),
\]
so it remains to find this common value from the type $\lambda$, for which it is worth understanding the structure behind $f(g)$.

If $g$ fixes a partition $P = \{B_1,B_2,B_3\}$, then $g$ takes blocks to blocks: the image of each $B_i$ is some $B_j$. Since there are only three blocks, this correspondence (a permutation of the blocks induced by $g$)
can only be of three types: fixing all three, swapping two and fixing the third, or permuting them cyclically. An example of each case is given below.

\begin{ex}\label{ex:ABC}
	Consider $\mathcal{C} = \{\{1,4,7\},\{2,5,8\},\{3,6,9\}\}$ and \linebreak $\mathcal{L}=\{\{1,2,3\},\{4,5,6\},\{7,8,9\}\}$.
	\begin{itemize}
		\item (Blocks individually fixed). For $g = (1\ 2)(4\ 5)$, the partition $\mathcal{L}$ is fixed: each block of $\mathcal{L}$ is a union of complete cycles of $g$. For example, $\{1,2,3\}$ gathers the cycle $(1\ 2)$ and the fixed point $3$.
		
		\item (Two blocks swapped). For $g = (1\ 2)(4\ 5)(7\ 8)$, we have $g(\{1,4,7\}) = \{2,5,8\}$ and $g(\{2,5,8\}) = \{1,4,7\}$, while $g(\{3,6,9\}) = \{3,6,9\}$. Hence $g$ fixes $\mathcal{C}$.
		
		\item (All three blocks permuted cyclically). For \linebreak $g = (1\,2\,3)(4\,5\,6)(7\,8\,9)$, we have $\{1,4,7\} \to \{2,5,8\} \to \{3,6,9\} \to \{1,4,7\}$. Hence $g$ also fixes $\mathcal{C}$.
	\end{itemize}
\end{ex}

These are exactly the three cases that can occur, each contributing separately to $f(g)$:

\begin{lema}\label{lema:fg}
	For $g \in S_9$,
	\[
	f(g) \;=\; A(g) + B(g) + C(g),
	\]
	where $A(g)$, $B(g)$, and $C(g)$ are, respectively, the number of 	partitions fixed by $g$ in which the blocks are individually fixed, in which two blocks are swapped with each other, and in which the three blocks are permuted cyclically.
\end{lema}
\begin{proof}
	If $g(P) = P$, the image of each block is a block, and $g$ induces a permutation $\bar g$ of the three blocks. Each fixed partition falls into exactly one of these cases (the case is determined by $\bar g$), so the three counts add up without overlap. 
\end{proof}

It remains to obtain, for each case, an explicit formula in terms of the cycle type of $g$. 
	
	\underline{Case $\bar g = \mathrm{id}$.} Each block is a set of $3$ elements invariant under $g$, and an invariant set is always a union of complete cycles of $g$, as in Example~\ref{ex:ABC}, where $\{1,2,3\}$ gathers the cycle $(1\ 2)$ and the fixed point $3$. Splitting the
	$9$ digits into three invariant blocks is, therefore, the same as splitting the cycles of $g$ into three groups of total length $3$: a single $3$-cycle, a $2$-cycle plus a fixed point, or three fixed points. This gives exactly $A(g)$ partitions. In particular, any cycle of length greater than $3$ makes $A(g) = 0$.
	
	\underline{Case $\bar g$ a transposition.} Say $g$ fixes $B_1$ and swaps $B_2 \leftrightarrow B_3$, as in Example~\ref{ex:ABC}. Block $B_1$ is invariant, so it is a union of cycles of $g$ of total length $3$.
	
	It remains to determine $B_2$ and $B_3$. Together they form an invariant set of $6$ elements, a union of cycles of $g$; the question is how these $6$ elements split between the two blocks. For illustration, take $g = (1\ 2)(4\ 5)(7\ 8)$ itself, from Example~\ref{ex:ABC}, with $B_1 = \{3,6,9\}$. The cycle $(1\ 2)$ links $1$ and $2$: since $g$ swaps $B_2$ and $B_3$, if $1$ is in $B_2$ then $g(1)=2$ must be in $B_3$, and vice versa. That is, the cycle $(1\ 2)$ offers exactly two ways of splitting its elements between $B_2$ and $B_3$ ($1\in B_2, 2\in B_3$, or the opposite), and the same holds, independently, for the cycles $(4\ 5)$ and $(7\ 8)$. This already reveals why only cycles of even length can take part in this alternation.
	
	In general, a cycle of even length $2m$ splits into two ``alternation classes'' of $m$ elements each. In Example~\ref{ex:ABC}, the cycle $(1\ 2)$ (with $m=1$) has the classes $\{1\}$ and $\{2\}$, and the choice of which class goes to $B_2$ is free and independent for each cycle. With $k$ cycles, this gives $2^k$ choices. In Example~\ref{ex:ABC}, $k=3$ and $2^3 = 8$. But each choice and its opposite (swapping all classes simultaneously) produce the same partition, only with the
	names $B_2$ and $B_3$ reversed: the $8$ choices group into $4$ pairs, giving $2^{k-1}$ distinct partitions for this fixed block. In the example, these are the $4$ partitions:
	\[
	\{1,4,7\}\{2,5,8\}\{3,6,9\}, \quad
	\{1,4,8\}\{2,5,7\}\{3,6,9\}, \]
\[	\{1,5,7\}\{2,4,8\}\{3,6,9\}, \quad
	\{1,5,8\}\{2,4,7\}\{3,6,9\}.
	\]
	It remains to vary which set of cycles forms the fixed block $B_1$: for each choice of cycles of total length $3$ such that all remaining cycles have even length, there are $2^{k-1}$ partitions, where $k$ is the number of those remaining cycles. Summing over all such choices gives $B(g)$.

\underline{Case $\bar g$ a $3$-cycle.} Say $g$ takes $B_1 \to B_2 \to B_3 \to B_1$, as in Example~\ref{ex:ABC}, with $g = (1\,2\,3)(4\,5\,6)(7\,8\,9)$.

Take the cycle $(1\,2\,3)$. Since $g$ takes $B_1$ to $B_2$, to $B_3$, and back to $B_1$, the digit $1$ and its image $g(1)=2$ cannot be in the same block: if $1 \in B_1$, then $2 = g(1) \in B_2$, $3 = g(2) \in B_3$, and $g(3)=1\in B_1$. That is, the entire cycle $(1\,2\,3)$ splits, in order, among $B_1$, $B_2$, and $B_3$: one ``alternation class'' for each block (now with period 3), here of a single element each. Once we choose which digit of the cycle goes to $B_1$ (three options: $1$, $2$, or $3$), the other two are determined. The same reasoning, independently, applies to the cycles $(4\,5\,6)$ and $(7\,8\,9)$.

In general, a cycle of length $3m$ splits into three alternation classes of $m$ elements each. In Example~\ref{ex:ABC}, cycles of length $3$ have singleton classes, and the choice of which class goes
to $B_1$ is free and independent for each cycle. That is why every cycle of $g$ needs to have length divisible by $3$, and, if any does not, $C(g) = 0$. Since $9$ can only be written as a sum of multiples of $3$ in three ways, $3{+}3{+}3$, $6{+}3$, and $9$, these are the only types with $C(g) \neq 0$.

With $k$ cycles, there are $3^k$ choices, and each determines an ordered triple $(B_1,B_2,B_3)$. In the example, $k=3$ and $3^3 = 27$ triples. But the partition does not distinguish which block we call ``$B_1$'': cyclically permuting the labels $B_1,B_2,B_3$ ($3$ ways) always gives the same set of three blocks, only relabeled. Thus, each partition is counted three times in this list of $27$, leaving $3^{k-1}$ distinct partitions. In the example, $27/3 = 9$; they are
\begin{gather*}
	\{1,4,7\}\{2,5,8\}\{3,6,9\}, \;
	\{1,4,8\}\{2,5,9\}\{3,6,7\}, \;
	\{1,4,9\}\{2,5,7\}\{3,6,8\}, \\
	\{1,5,7\}\{2,6,8\}\{3,4,9\}, \;
	\{1,5,8\}\{2,6,9\}\{3,4,7\}, \;
	\{1,5,9\}\{2,6,7\}\{3,4,8\}, \\
	\{1,6,7\}\{2,4,8\}\{3,5,9\}, \;
	\{1,6,8\}\{2,4,9\}\{3,5,7\}, \;
	\{1,6,9\}\{2,4,7\}\{3,5,8\}.
\end{gather*}
This same $g$ also fixes a tenth partition, $\mathcal{L}$, but with the three blocks individually invariant, the case $\bar g = \mathrm{id}$: it does not enter $C(g)$, but rather $A(g)$.

\paragraph{The table of $f$.} Applying Lemma~\ref{lema:fg} to each of the $30$ types, we obtain Table~\ref{tab:f-tipos}, where each row corresponds to a type $\lambda$, and the columns $A(g)$ through $f(g)$ are computed for any representative $g$ of $\lambda$, a value independent of the choice by Proposition~\ref{prop:tipo}. 

We present below three representative examples, each with a concrete permutation of the type.

\begin{itemize}
	\item \underline{Type $2{+}1^7$.} Say $g = (1\ 2)$. 
	
	Here $C(g) = 0$, since the lengths $2$ and $1$ are not divisible by $3$, and $B(g) = 0$, since any candidate fixed block of total length $3$ (either $\{2{+}1\}$ or $\{1^3\}$) leaves an odd cycle among the rest, violating the parity requirement.
	
	What remains is $A(g)$: an invariant partition needs to place $1$ and $2$ in the same block. The block $\{1,2,x\}$ has $7$ choices for $x$, and the remaining $6$ digits split into two blocks in $\binom{6}{3}/2 = 10$ ways: $A(g) = 7 \times 10 = 70$. 
	
	Hence, $f(g) = 70+0+0=70$.
	
	\item \underline{Type $2^4{+}1$.} Say $g = (1\ 2)(3\ 4)(5\ 6)(7\ 8)$. 
	
	Here $A(g) = 0$: invariant blocks would be $\{2{+}1\}$ or $\{1^3\}$, but there is a single fixed point, so only one of the four $2$-cycles could form a block, the others could not.
	
	Also $C(g) = 0$, since there are cycles in $g$ whose length is not a multiple of $3$. 
	
	For $B(g)$: the fixed block must be a $2$-cycle plus the fixed point $9$ ($4$ choices for the $2$-cycle); three $2$-cycles remain, all even, giving $2^{3-1} = 4$ partitions each. For example, with the fixed block $\{7,8,9\}$, each of the other two blocks takes one element from each remaining $2$-cycle, and the $4$ possibilities are:
	\[
	\{1,3,5\}\{2,4,6\}, \; \{1,3,6\}\{2,4,5\}, \;
	\{1,4,5\}\{2,3,6\}, \; \{1,4,6\}\{2,3,5\}.
	\]
	In all, $B(g) = 4 \times 4 = 16$. 
	
	Hence, $f(g) = 0+16+0=16$.
	
	\item \underline{Type $6{+}3$.} Say
	$g = (1\,2\,3\,4\,5\,6)(7\,8\,9)$. 
	
	Here $A(g) = 0$: the $6$-cycle does not fit in any block. 
	
	For $B(g)$: the fixed block can only be the $3$-cycle, $\{7,8,9\}$. The $6$-cycle remains, even, with $2^{1-1} = 1$ partition. Its two alternation classes (period $2$) are $\{1,3,5\}$ and $\{2,4,6\}$. 
	
	For $C(g)$: both lengths are divisible by $3$ and there are $k=2$ cycles, giving $3^{2-1} = 3$ partitions. Now, each block takes an alternation class of the $6$-cycle, of period $3$ ($\{1,4\}$, $\{2,5\}$, or $\{3,6\}$), plus one digit of the $3$-cycle, and the choice of which digit accompanies $\{1,4\}$ determines the rest. 
	
	Hence $f(g) = 0+1+3 = 4$. The complete list follows:
\begin{gather*}
	\{1,3,5\}\{2,4,6\}\{7,8,9\}, \quad
	\{1,4,7\}\{2,5,8\}\{3,6,9\}, \\
	\{1,4,8\}\{2,5,9\}\{3,6,7\}, \quad
	\{1,4,9\}\{2,5,7\}\{3,6,8\},
\end{gather*}
	the first coming from case $B(g)$ and the last three from case $C(g)$. 
\end{itemize}

\begin{table}
	\centering
	\renewcommand{\arraystretch}{1.12}
	\small
	\begin{tabular}{|l|r|r|r|r|r|}
		\hline
		Type $\lambda$ of $g$ & $|S_\lambda|$ & $A(g)$ & $B(g)$ & $C(g)$ & $f(g)$ \\
		\hline
		$1^9$               & $1$      & $280$ & $0$  & $0$ & $280$ \\
		$2{+}1^7$           & $36$     & $70$  & $0$  & $0$ & $70$  \\
		$2^2{+}1^5$         & $378$    & $20$  & $0$  & $0$ & $20$  \\
		$3{+}1^6$           & $168$    & $10$  & $0$  & $0$ & $10$  \\
		$2^3{+}1^3$         & $1,260$  & $6$   & $4$  & $0$ & $10$  \\
		$3{+}2{+}1^4$       & $2,520$  & $4$   & $0$  & $0$ & $4$   \\
		$4{+}1^5$           & $756$    & $0$   & $0$  & $0$ & $0$   \\
		$2^4{+}1$           & $945$    & $0$   & $16$ & $0$ & $16$  \\
		$3{+}2^2{+}1^2$     & $7,560$  & $2$   & $0$  & $0$ & $2$   \\
		$3^2{+}1^3$         & $3,360$  & $1$   & $0$  & $0$ & $1$   \\
		$4{+}2{+}1^3$       & $7,560$  & $0$   & $2$  & $0$ & $2$   \\
		$5{+}1^4$           & $3,024$  & $0$   & $0$  & $0$ & $0$   \\
		$3{+}2^3$           & $2,520$  & $0$   & $4$  & $0$ & $4$   \\
		$3^2{+}2{+}1$       & $10,080$ & $1$   & $0$  & $0$ & $1$   \\
		$4{+}2^2{+}1$       & $11,340$ & $0$   & $4$  & $0$ & $4$   \\
		$4{+}3{+}1^2$       & $15,120$ & $0$   & $0$  & $0$ & $0$   \\
		$5{+}2{+}1^2$       & $18,144$ & $0$   & $0$  & $0$ & $0$   \\
		$6{+}1^3$           & $10,080$ & $0$   & $1$  & $0$ & $1$   \\
		$3^3$               & $2,240$  & $1$   & $0$  & $9$ & $10$  \\
		$4{+}3{+}2$         & $15,120$ & $0$   & $2$  & $0$ & $2$   \\
		$4^2{+}1$           & $11,340$ & $0$   & $0$  & $0$ & $0$   \\
		$5{+}2^2$           & $9,072$  & $0$   & $0$  & $0$ & $0$   \\
		$5{+}3{+}1$         & $24,192$ & $0$   & $0$  & $0$ & $0$   \\
		$6{+}2{+}1$         & $30,240$ & $0$   & $1$  & $0$ & $1$   \\
		$7{+}1^2$           & $25,920$ & $0$   & $0$  & $0$ & $0$   \\
		$5{+}4$             & $18,144$ & $0$   & $0$  & $0$ & $0$   \\
		$6{+}3$             & $20,160$ & $0$   & $1$  & $3$ & $4$   \\
		$7{+}2$             & $25,920$ & $0$   & $0$  & $0$ & $0$   \\
		$8{+}1$             & $45,360$ & $0$   & $0$  & $0$ & $0$   \\
		$9$                 & $40,320$ & $0$   & $0$  & $1$ & $1$   \\
		\hline
	\end{tabular}
	\caption{The $30$ types in $S_9$, with the size $|S_\lambda|$ of each class given by \eqref{eq:S_lambda} and the decomposition $f(g) = A(g) + B(g) + C(g)$ from Lemma~\ref{lema:fg}.}
	\label{tab:f-tipos}
\end{table}

It is worth understanding the pattern of zeros in Table~\ref{tab:f-tipos}, which has eleven rows with $f(g) = 0$. A cycle of length $5$, $7$, or $8$ zeroes out all three terms at once: it does not fit in an invariant block (length greater than $3$, zeroing $A(g)$); it cannot take part in the alternation between two blocks ($5$ and $7$ because they are odd, $7$ and $8$ also because they exceed the $6$ elements available in the two blocks), zeroing $B(g)$; and none of the three is divisible by $3$, zeroing $C(g)$. This eliminates the eight types that contain one of these lengths. The three remaining null rows, $4{+}1^5$, $4{+}3{+}1^2$, and $4^2{+}1$, fall for other reasons: the $4$ zeroes $A(g)$; in the analysis of $B(g)$, every choice of fixed block leaves behind some odd cycle (a fixed point or the $3$-cycle), or there is no possible choice at all; and no type has all its lengths divisible by $3$, zeroing $C(g)$.

Note that the sum $\sum_\lambda |S_\lambda| \times f(g)$ over Table~\ref{tab:f-tipos} gives exactly
$362,880 = 9!$. This is what Burnside's Lemma predicts when applied to $S_9$ acting on $X$ alone, since the $280$ partitions form a single orbit under relabeling.

\subsubsection{The powers and the final sum}

The formula from Proposition~\ref{prop:adapB} also calls for $f(g^2)$ and $f(g^3)$, read off the same Table~\ref{tab:f-tipos}. A classical fact about powers is useful: the $m$-th power of a cycle of length $k$ decomposes into $\mathrm{gcd}(k,m)$ cycles of length $k/\mathrm{gcd}(k,m)$. For example, $(1\,2\,3\,4\,5\,6)^2= (1\,3\,5)(2\,4\,6)$ and $(1\,2\,3\,4\,5\,6)^3=(1\,4)(2\,5)(3\,6)$.

Applying this fact cycle by cycle, the type of $g^m$ can be read directly from the type of $g$. For example, if $g$ has type $6{+}2{+}1$, then $g^2$ has type $3^2{+}1^3$ (the $6$-cycle becomes two $3$-cycles; the $2$-cycle, two fixed points) and $g^3$ has type $2^4{+}1$ (the $6$-cycle becomes three $2$-cycles; the $2$-cycle stays the same).

Before summing, one observation saves work: the term for a type is $f(g)^3 + 3 f(g) f(g^2) + 2 f(g^3)$, which vanishes whenever $f(g) = 0$ and $f(g^3) = 0$. This immediately discards $11$ of the $30$ types: all those that contain a cycle of length $5$, $7$, or $8$, plus $4{+}1^5$, $4{+}3{+}1^2$, and $4^2{+}1$ (whose cubes have types $4{+}1^5$, $4{+}1^5$, and $4^2{+}1$, all with $f = 0$). This leaves $19$ terms, gathered in Table~\ref{tab:soma-final}.

\begin{table}
	\centering
	\renewcommand{\arraystretch}{1.12}
	\footnotesize
	\begin{tabular}{|l|r|l|l|r|r|r|r|}
		\hline
		Type $\lambda$ of $g$ & $|S_\lambda|$ & Type of $g^2$ & Type of $g^3$ &
		$f(g)$ & $f(g^2)$ & $f(g^3)$ & Term \\
		\hline
		$1^9$           & $1$      & $1^9$           & $1^9$           & $280$ & $280$ & $280$ & $22,187,760$ \\
		$2{+}1^7$       & $36$     & $1^9$           & $2{+}1^7$       & $70$  & $280$ & $70$  & $14,469,840$ \\
		$2^2{+}1^5$     & $378$    & $1^9$           & $2^2{+}1^5$     & $20$  & $280$ & $20$  & $9,389,520$  \\
		$3{+}1^6$       & $168$    & $3{+}1^6$       & $1^9$           & $10$  & $10$  & $280$ & $312,480$    \\
		$2^3{+}1^3$     & $1,260$  & $1^9$           & $2^3{+}1^3$     & $10$  & $280$ & $10$  & $11,869,200$ \\
		$3{+}2{+}1^4$   & $2,520$  & $3{+}1^6$       & $2{+}1^7$       & $4$   & $10$  & $70$  & $816,480$    \\
		$2^4{+}1$       & $945$    & $1^9$           & $2^4{+}1$       & $16$  & $280$ & $16$  & $16,601,760$ \\
		$3{+}2^2{+}1^2$ & $7,560$  & $3{+}1^6$       & $2^2{+}1^5$     & $2$   & $10$  & $20$  & $816,480$    \\
		$3^2{+}1^3$     & $3,360$  & $3^2{+}1^3$     & $1^9$           & $1$   & $1$   & $280$ & $1,895,040$  \\
		$4{+}2{+}1^3$   & $7,560$  & $2^2{+}1^5$     & $4{+}2{+}1^3$   & $2$   & $20$  & $2$   & $997,920$    \\
		$3{+}2^3$       & $2,520$  & $3{+}1^6$       & $2^3{+}1^3$     & $4$   & $10$  & $10$  & $514,080$    \\
		$3^2{+}2{+}1$   & $10,080$ & $3^2{+}1^3$     & $2{+}1^7$       & $1$   & $1$   & $70$  & $1,451,520$  \\
		$4{+}2^2{+}1$   & $11,340$ & $2^2{+}1^5$     & $4{+}2^2{+}1$   & $4$   & $20$  & $4$   & $3,538,080$  \\
		$6{+}1^3$       & $10,080$ & $3^2{+}1^3$     & $2^3{+}1^3$     & $1$   & $1$   & $10$  & $241,920$    \\
		$3^3$           & $2,240$  & $3^3$           & $1^9$           & $10$  & $10$  & $280$ & $4,166,400$  \\
		$4{+}3{+}2$     & $15,120$ & $3{+}2^2{+}1^2$ & $4{+}2{+}1^3$   & $2$   & $2$   & $2$   & $362,880$    \\
		$6{+}2{+}1$     & $30,240$ & $3^2{+}1^3$     & $2^4{+}1$       & $1$   & $1$   & $16$  & $1,088,640$  \\
		$6{+}3$         & $20,160$ & $3^3$           & $2^3{+}1^3$     & $4$   & $10$  & $10$  & $4,112,640$  \\
		$9$             & $40,320$ & $9$             & $3^3$           & $1$   & $1$   & $10$  & $967,680$    \\
		\hline
		\multicolumn{7}{|r|}{\textbf{Sum}} & $\mathbf{95,800,320}$ \\
		\hline
	\end{tabular}
	\caption{The $19$ non-zero terms of the formula from Proposition~\ref{prop:adapB}. Each term is $|S_\lambda| \times \big[f(g)^3 + 3 f(g) f(g^2) + 2 f(g^3)\big]$, with all values of $f$ read from Table~\ref{tab:f-tipos} via the types of $g^2$ and $g^3$.}
	\label{tab:soma-final}
\end{table}

The final sum is $95,800,320$, so, by Proposition~\ref{prop:adapB},
\[
N \;=\; \frac{95,800,320}{9! \times 6} \;=\; \frac{95,800,320}{2,177,280}
\;=\; 44.
\]
This completes a closed proof of Theorem~\ref{teo:44}: the $44$ classes of the Sudoku enumeration, originally obtained as a byproduct of a cascade of computational reductions, are the value
of a sum of $19$ terms, each computable by hand from Lemma~\ref{lema:fg}.\footnote{All values in Tables~\ref{tab:f-tipos} and~\ref{tab:soma-final} were independently checked by computational brute force: the decomposition $A(g)+B(g)+C(g)$ against the direct count of fixed partitions per type, and the final sum against a direct evaluation of the formula from Proposition~\ref{prop:adapB} over the $9!$ elements of $S_9$, with $f$ computed via a third, independent route (the stabilizers of the $280$ partitions).}

The same method, applied analogously to the group $\mathrm{Stab}(\mathcal{C}) \le S_9$ acting on ordered and unordered pairs, closes the earlier stages of the cascade as well: the $131$ and $84$ counts from Section~\ref{sec:derivacao} also admit a derivation by hand, with the contributing cycle types being exactly the $19$ types with a non-zero term in Table~\ref{tab:soma-final}. We omit the details.


\section*{Use of artificial intelligence tools}

While preparing this work, the author used the artificial intelligence assistant Claude (Claude Sonnet~5 model, accessed via claude.ai) for language editing and translation, for assistance in deriving and developing some of the mathematical arguments, and for writing the Python scripts used in the computational verifications described throughout the paper. The author reviewed, verified, and complemented all content produced with the assistance of this tool, independently checked all computational verifications and all mathematical claims, and takes full responsibility for the content of this publication.



\begin{thebibliography}{99}
	
\bibitem{Adler2008}
A. Adler and I. Adler.
\textit{Fundamental Transformations of Sudoku Grids}.
Mathematical Spectrum, v. 41, n. 1, p. 2--7, 2008/2009.
	

\bibitem{Berthier2007}
D. Berthier.  
\textit{The Hidden Logic of Sudoku}. Lulu Press, 2007.


\bibitem{Delahaye2006}
J. Delahaye.
\textit{The Science Behind Sudoku}.
Scientific American, v. 294, n. 6, p. 80--87, 2006.


\bibitem{FelgenhauerJarvis2006}
B. Felgenhauer and F. Jarvis.
\textit{Mathematics of Sudoku I}.
Mathematical Spectrum, v. 39, n. 1, p. 15--22, 2006.

\bibitem{Hayes2006}
B. Hayes.
\textit{Unwed Numbers: The Mathematics of Sudoku}.
American Scientist, v. 94, p. 12--15, 2006. 
Available at: \url{https://www.americanscientist.org/article/unwed-numbers}. 
Accessed on: Jul. 21, 2026.


\bibitem{siteJarvis}
F. Jarvis. \textit{The 44 classes in counting Sudoku grids}.
Available at: \url{http://www.afjarvis.org.uk/sudoku/ed44.html}. 2005.
Accessed on: Jul. 21, 2026.

\bibitem{siteJarvis2}
F. Jarvis. \textit{There are 5472730538 essentially different Sudoku grids...
	and the Sudoku symmetry group}. 2005.
Available at: \url{http://www.afjarvis.org.uk/sudoku/sudgroup.html}.
Accessed on: Jul. 21, 2026.


\bibitem{Rosenhouse2011}
J. Rosenhouse and L. Taalman.  
\textit{Taking Sudoku Seriously: The Math Behind the World’s Most Popular Pencil Puzzle}.  
Oxford University Press, 2011.


\bibitem{Russel2006}
E. Russel and F. Jarvis.
\textit{Mathematics of Sudoku II}.
Mathematical Spectrum, v. 39, n. 2, p. 54--58, 2006/2007.



\bibitem{Shi}
F. Shi, M. Zhang, and H. Aslaksen. 
\textit{Enumerating $9\times 9$ Sudoku grids}. 2014. Available at \url{https://www.researchgate.net/publication/251390470}. Accessed on: Jul. 21, 2026.



\end{thebibliography}
\end{document}